\documentclass[a4paper,12pt]{article}
\usepackage{amssymb,amsmath,amsthm,times,mathrsfs,fullpage}
\usepackage[pdftex]{hyperref}
\hypersetup{plainpages=True, pdfstartview=FitH, bookmarksopen=true,
colorlinks=true,linkcolor=blue,citecolor=blue}
\usepackage{tikz,pgf}
\usetikzlibrary{patterns,arrows,decorations.pathreplacing}
\usepackage{epstopdf}
\usepackage[nospace,noadjust]{cite}

\begin{document}

\newtheorem{thm}{Theorem}[section]
\newtheorem{cor}[thm]{Corollary}
\newtheorem{lmm}[thm]{Lemma}
\newtheorem{conj}[thm]{Conjecture}
\newtheorem{pro}[thm]{Proposition}
\theoremstyle{definition}\newtheorem{df}[thm]{Definition}
\theoremstyle{remark}\newtheorem{rem}[thm]{Remark}

\newcommand*\samethanks[1][\value{footnote}]{\footnotemark[#1]}
\newcommand*{\myand}{\,\and\,}

\title{\textbf{Galled Tree-Child Networks}}
\author{Yu-Sheng Chang\thanks{Department of Mathematical Sciences, National Chengchi University, Taipei 116, Taiwan.}
\myand
Michael Fuchs\samethanks
\myand
Guan-Ru Yu\thanks{Department of Applied Mathematics, National  Sun Yat-sen University, Kaohsiung 804, Taiwan.}}
\date{\today}
 \maketitle
 \vspace*{-0.2cm}
\begin{center}
{\it Dedicated to Hsien-Kuei Hwang on the occasion of his 60th birthday}
\end{center}
\begin{abstract}
We propose the class of {\it galled tree-child networks} which is obtained as intersection of the classes of galled networks and tree-child networks. For the latter two classes, (asymptotic) counting results and stochastic results have been proved with very different methods. We show that a counting result for the class of galled tree-child networks follows with similar tools as used for galled networks, however, the result has a similar pattern as the one for tree-child networks. In addition, we also consider the (suitably scaled) numbers of reticulation nodes of random galled tree-child networks and show that they are asymptotically normal distributed. This is in contrast to the limit laws of the corresponding quantities for galled networks and tree-child networks which have been both shown to be discrete.
\end{abstract}

\section{Introduction}
Phylogenetic networks are used to visualize, model, and analyze the ancestor relationship of taxa in reticulate evolution. To make them more relevant for biological applications as well as devise algorithms for them, many subclasses of the class of phylogenetic networks have been proposed; see the comprehensive survey \cite{KoPoKuWi}. A lot of recent research work was concerned with fundamental questions such as counting them and understanding the shape of a network drawn uniformly at random from a given class; see, e.g., \cite{BoGaMa,CaZh,ChFu,FuGiMa1,FuGiMa2,FuYuZh1,FuYuZh2,FuHuYu,GuRaZh,Ma,PoBa}. Despite of this, even counting results are still missing for most of the major classes of phylogenetic networks. Two notable exceptions are tree-child networks and galled networks for which such results have been proved in \cite{FuYuZh1,FuYuZh2}. In this work, we consider the intersection of these two network classes. We start with some basic definitions and then explain why we find this class interesting.

First, a phylogenetic network is defined as follows.
\begin{df}[Phylogenetic Network]
A (rooted) phylogenetic network of size $n$ is a rooted, simple, directed, acyclic graph whose nodes fall into the following three (disjoint) categories:
\begin{enumerate}
\item[(a)] A unique {\it root} which has indegree $0$ and outdegree $1$;
\item[(b)] {\it Leaves} which have indegree $1$ and outdegree $0$ and are bijectively labeled with labels from the set $\{1,\ldots, n\}$;
\item[(c)] {\it Internal nodes} which have indegree and outdegree at least $1$ and total degree at least $3$.
\end{enumerate}
Moreover, a phylogenetic network is called {\it binary} if (c) is replaced by
\begin{enumerate}
\item[(c')] {\it Internal nodes} which have either indegree $1$ and outdegree $2$ ({\it tree nodes}) or indegree $2$ and outdegree $1$ ({\it reticulation nodes}).
\end{enumerate}
\end{df}
\begin{rem}
\begin{itemize}
\item[(i)] Phylogenetic networks with all internal nodes having indegree equal to $1$ are called {\it phylogentic trees}.
\item[(ii)] If not explicitly mentioned, phylogenetic networks will always be binary in the sequel.
\end{itemize}
\end{rem}

We next define galled networks and tree-child networks which are two of the major classes of phylogenetic networks. For the definition, we need the notion of a {\it tree cycle} which is a pair of edge-disjoint paths in a phylogenetic network that start at a common tree node and end at a common reticulation node with all other nodes being tree nodes.

\begin{df}
\begin{itemize}
\item[(a)] A phylogenetic network is called a {\it tree-child network} if every non-leaf node has at least one child which is either a reticulation node or a leaf.
\item[(b)] A phylogenetic network is called a {\it galled network} if every reticulation node is in a (necessarily unique) tree cycle.
\end{itemize}
\end{df}

\begin{rem}\label{tc-gn-diff} Note that neither the class of tree-child networks is contained in the class of galled networks nor vice versa.
\end{rem}

Let ${\rm TC}_{n,k}$ and ${\rm GN}_{n,k}$ denote the number of tree-child networks and galled networks of size $n$ with $k$ reticulation nodes, respectively. It is not hard to see that $k\leq n-1$ for tree-child networks and $k\leq 2n-2$ for galled networks where both bounds are sharp. Thus, the total numbers are given by:
\begin{equation}\label{total-numbers}
{\rm TC}_n:=\sum_{k=0}^{n-1}{\rm TC}_{n,k}\qquad\text{and}\qquad{\rm GN}_{n}:=\sum_{k=0}^{2n-2}{\rm GN}_{n,k}.
\end{equation}
The asymptotic growth of both of these sequences is known. First, in \cite{FuYuZh1}, it was proved that for the number of tree-child networks, as $n\rightarrow\infty$,
\begin{equation}\label{asymp-tcn}
{\rm TC}_n=\Theta\left(n^{-2/3}e^{a_1(3n)^{1/3}}\left(\frac{12}{e^2}\right)^n n^{2n}\right),
\end{equation}
where $a_1$ is the largest root of the Airy function of the first kind. The surprise here was the presence of a {\it stretched exponential} in the asymptotic growth term. On the other hand, no stretched exponential is contained in the asymptotics of the number of galled networks. More precisely, it was proved in \cite{FuYuZh2} that
\begin{equation}\label{asymp-gnn}
{\rm GN}_n\sim\frac{\sqrt{2e\sqrt[4]{e}}}{4}n^{-1}\left(\frac{8}{e^2}\right)^n n^{2n}.
\end{equation}

The tools used to establish (\ref{asymp-tcn}) and (\ref{asymp-gnn}) were very different: for (\ref{asymp-tcn}), a bijection to a class of words was proved and a recurrence for these word was found which could be (asymptotically) analyzed with the approach from \cite{ElFaWa}; for (\ref{asymp-gnn}), the component graph method introduced in \cite{GuRaZh} together with the Laplace method and a result from \cite{BeRi} was used.

\newpage

Another difference was the location in (\ref{total-numbers}) of the terms which dominate the two sums. For tree-child networks, the main contribution comes from networks with $k$ close to $n-1$ (the maximal-reticulated networks), whereas for galled networks, the main contributions comes from networks with $k\approx n$. In fact, the limit law of the number of reticulation nodes, say $R_n$, was derived in \cite{ChFuLiWaYu,FuYuZh2} for both network classes if a network of size $n$ is sampled uniformly at random. More precisely, for tree-child networks, it was shown in \cite{ChFuLiWaYu} that, as $n\rightarrow\infty$,
\[
n-1-R_n\stackrel{d}{\longrightarrow}{\rm Poisson}(1/2),
\]
where $\stackrel{d}{\longrightarrow}$ denotes convergence in distribution and ${\rm Poisson}(\lambda)$ is a Poisson law with parameter $\lambda$. A similar discrete limit law was proved in \cite{FuYuZh2} for galled networks, however, the limit law is not Poisson but a mixture of Poisson laws; see Theorem~2 in \cite{FuYuZh2} for details.

Due to the above results and differences, one wonders how the intersection of the class of tree-child networks and galled networks behaves?

\begin{df}[Galled Tree-Child Network]
A {\it galled tree-child network} is a network which is both a galled network and a tree-child network.
\end{df}

Let ${\rm GTC}_{n,k}$ denote the number of galled tree-child networks of size $n$ with $k$ reticulation nodes. We will show below that $k\leq n-1$. (See Lemma~\ref{finite-k} in Section~\ref{small-large-k}.) Set:
\[
{\rm GTC}_{n}:=\sum_{k=0}^{n-1}{\rm GTC}_{n,k}.
\]
Then, this sequence has the following first-order asymptotics.

\begin{thm}\label{main-thm-1}
For the number of galled tree-child networks, we have, as $n\rightarrow\infty$,
\[
{\rm GTC}_n\sim\frac{1}{2\sqrt[4]{e}}n^{-5/4}{e}^{2 \sqrt{n}}\left(\frac{2}{e^2}\right)^n n^{2n}.
\]
\end{thm}
\begin{rem}
Note that the asymptotics contains a stretched exponential as does the expansion (\ref{asymp-tcn}) for tree-child networks, however, the proof will use the tools which were used in \cite{FuYuZh2} to derive (\ref{asymp-gnn}) for galled networks.
\end{rem}

We next consider the number of reticulation nodes $R_n$ of a {\it random galled tree-child network} which is a galled tree-child network of size $n$ that is sampled uniformly at random from the set of all galled tree-child networks of size $n$. In contrast to tree-child networks and galled networks, the limit law of $R_n$ (suitably scaled) is continuous.

\begin{thm}\label{main-thm-2}
The number of reticulation nodes $R_n$ of a random galled tree-child networks satisfies, as $n\rightarrow\infty$,
\[
\frac{R_n-{\mathbb E}(R_n)}{\sqrt{{\rm Var}(R_n)}}\stackrel{d}{\longrightarrow} N(0,1),
\]
where $N(0,1)$ denotes the standard normal distribution. Moreover,
\[
{\mathbb E}(R_n)\sim n-\sqrt{n}\qquad\text{and}\qquad {\rm Var}(R_n)\sim \sqrt{n}/2.
\]
\end{thm}

The above results show that galled tree-child networks behave quite different from both tree-child networks and galled networks. That is one reason why we find them interesting.

\newpage

Another reason stems from a recent result which was proved in \cite{ChFu}. In the latter paper, the asymptotics of ${\rm GN}_{n,k}$ for fixed $k$ was derived. Let ${\rm PN}_{n,k}$ denote the number of phylogenetic networks of size $n$ and $k$ reticulation nodes. (Note that this number is finite, whereas it becomes infinite when summing over $k$.) Then, one of the main results from \cite{ChFu} implies that for fixed $k$, as $n\rightarrow\infty$,
\begin{equation}\label{asymp-fixed-k}
{\rm PN}_{n,k}\sim{\rm TC}_{n,k}\sim{\rm GN}_{n,k}\sim\frac{2^{k-1}\sqrt{2}}{k!}\left(\frac{2}{e}\right)^nn^{n+2k-1}.
\end{equation}
(The first two asymptotic equivalences were proved in \cite{FuHuYu,Ma}.) That ${\rm TC}_{n,k}$ and ${\rm GN}_{n,k}$ have the same first-order asymptotics for fixed $k$ was a surprise since the classes of tree-child networks and galled networks are quite different, e.g., neither contains the other; see  Remark~\ref{tc-gn-diff}. However, the above result can be explained via the class of galled tree-child networks as will be seen in Section~\ref{small-large-k} below.

We conclude the introduction with a short sketch of the paper. The proofs of Theorem~\ref{main-thm-1} and Theorem~\ref{main-thm-2} follow with a similar approach as used for galled networks in \cite{FuYuZh1}. This approach is based on the component graph method from \cite{GuRaZh} which we will recall in the next section. Then, in Section~\ref{small-large-k}, we will consider ${\rm GTC}_{n,k}$ for small and large values of $k$. Finally, Section~\ref{proof-main-thm} will contain the proofs of our main results (Theorem~\ref{main-thm-1} and Theorem~\ref{main-thm-2}). We will conclude the paper with some final remarks in Section~\ref{con}.

\section{The Component Graph Method}\label{comp-graph-method}

\begin{figure}[t!]
\begin{center}
\includegraphics[scale=1.28]{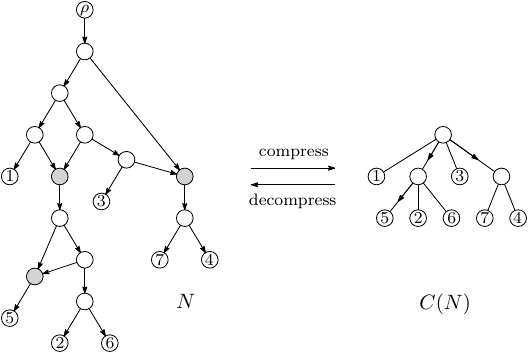}
\end{center}
\caption{A galled network $N$ and its component graph $C(N)$ which is a phylogenetic tree.}\label{galled-comp-graph}
\end{figure}

The component graph method for galled networks was introduced in \cite{GuRaZh} and used in \cite{ChFu,FuYuZh2} to prove asymptotic results. It is explained in detail in all these papers. However, to make the current paper self-contained, we will briefly recall it.

Let $N$ be a galled network. Then, by removing all the edges leading to reticulation vertices (these are the so-called {\it reticulation edges}), we obtain a forest whose trees are called the {\it tree-components} of $N$.

The {\it component graph} of $N$, denoted by $C(N)$, is now a rooted, directed, acyclic graph which has a vertex for every tree-component. Moreover, the vertices are connected in the same way as the tree-components have been connected by the reticulation edges. Finally, we attach the leaves in the tree-components to the corresponding vertices in $C(N)$ unless a vertex $v$ of $C(N)$ is a terminal vertex and its corresponding tree-component has exactly one leaf, in which case we use the label of that leaf to label $v$. Note that every non-root vertex in $C(N)$ has indegree $2$ and that $C(N)$ may contain double edges. We will replace such a double edge by a single edge and indicate that it was a double edge by placing an arrow on it; see Figure~\ref{galled-comp-graph} for a galled network together with its component graph. Also, denote by $\tilde{C}(N)$ the component graph of $C(N)$ with all arrows on edges removed. Then, the authors of \cite{GuRaZh} made the following important observation.

\begin{pro}[\cite{GuRaZh}]\label{comp-graph-gn}
$N$ is a galled network if and only if $\tilde{C}(N)$ is a (not necessarily binary) phylogenetic tree.
\end{pro}

The component graph can be seen as a compression of $N$. Thus, in order to generate all galled networks of size $n$, one only needs to list all component graphs (i.e., phylogenetic trees) with $n$ labeled leaves and decompress them.

We will next explain the decompression procedure. For this, we need the notion of {\it one-component networks}.

\begin{df}[One-component Network]
A phylogenetic network is called a {\it one-component network} if every reticulation node has a leaf as its child.
\end{df}

Now, let $C(N)$ be given. We do a breadth-first traversal of $C(N)$ and replace every vertex $v$ by a one-component galled network $O_v$ whose leaves below reticulation vertices are labeled with the first $k$ labels, where $k$ is the number of outgoing edges of $v$ in $C(N)$ that have an arrow on them, and whose size is equal to the outdegree of $v$. Then, attach the subtrees of $v$ which are connected to $v$ by edges with arrows on them to the leaves of $O_v$ with labels $\{1,\ldots,k\}$, where the subtree with the smallest label is attached to $1$, the subtree with the second largest label is attached to $2$, etc. Moreover, relabel the remaining leaves of $O_v$ by the remaining subtrees of $v$ (which are all of size $1$, i.e., they are leaves) in an order-consistent way. By using all possible one-component galled networks in every step, this gives all possible galled networks with $C(N)$ as component graph. Moreover, if we start from $\tilde{C}(N)$, then we first have to place arrows on all edges whose heads are internal nodes of $C(N)$ and for all remaining edges, we can freely decide if we want to place an arrow on them or not. Overall, this gives the following result which was one of the main results of \cite{GuRaZh}.

\begin{pro}[\cite{GuRaZh}]\label{GNn-GuRaZh}
We have,
\[
{\rm GN}_n=\sum_{{\mathcal T}}\prod_{v}\sum_{j=0}^{c_{\rm lf}(v)}\binom{c_{\rm lf}(v)}{j}M_{c(v),c(v)-c_{\rm lf}(v)+j},
\]
where the first sum runs over all (not necessarily binary) phylogenetic trees ${\mathcal T}$ of size $n$, the product runs over all internal nodes of ${\mathcal T}$, $c(v)$ is the outdegree of $v$, $c_{\rm lf}(v)$ is the number of children of $v$ which are leaves, and $M_{n,k}$ denotes the number of one-component galled networks of size $n$ with $k$ reticulation vertices, where the leaves below the reticulation vertices are labeled with labels from the set $\{1,\ldots,k\}$.
\end{pro}

For galled tree-child networks, it is now clear that the same formula holds with the only difference that $M_{n,k}$ has to be replaced by the corresponding number of one-component galled tree-child networks. However, this number is the same as the number of one-component tree-child networks.

\begin{lmm}
Every one-component tree-child network is a one-component galled tree-child network.
\end{lmm}

\begin{proof}
Let $v$ be a reticulation vertex and consider a pair of edge-disjoint paths from a common tree vertex to $v$. (Note that such a pair trivially exists.) Then, no internal vertex can be a reticulation vertex because such a reticulation vertex would not be followed by a leaf. Thus, $v$ is in a tree cycle which shows that the network is indeed galled.
\end{proof}

Denote by ${\rm L}_{n,k}$ the number of one-component tree-child networks of size $n$ and $k$ reticulation vertices, where the labels of the leaves below the reticulation vertices are $\{1,\ldots,k\}$. Then, we have the following analogous result to Proposition~\ref{GNn-GuRaZh}.

\begin{pro}
We have,
\begin{equation}\label{formula-GTCn}
{\rm GTC}_n=\sum_{{\mathcal T}}\prod_{v}\sum_{j=0}^{c_{\rm lf}(v)}\binom{c_{\rm lf}(v)}{j}L_{c(v),c(v)-c_{\rm lf}(v)+j},
\end{equation}
where notation is as in Proposition~\ref{GNn-GuRaZh} and $L_{n,k}$ was defined above.
\end{pro}

\begin{rem}
Using this result, one can obtain the following table for small values of $n$:
\vspace*{0.2cm}
\begin{table}[h]
\begin{center}
\begin{tabular}{c|c}
$n$ & $\mathrm{GTC}_n$ \\
\hline
1 & 1 \\
2 & 3 \\
3 & 48 \\
4 & 1,611 \\
5 & 87,660 \\
6 & 6,891,615 \\
7 & 734,112,540 \\
8 & 101,717,195,895 \\
9 & 17,813,516,259,420 \\
10 & 3,857,230,509,496,875 \\
\end{tabular}
\end{center}
\caption{The values of $\mathrm{GTC}_n$ for $1\leq n\leq 10$.}\label{GN-n}
\end{table}
\end{rem}

We will deduce all our results from (\ref{formula-GTCn}). In addition, we will make use of the following results for $L_{n,k}$ which were proved in \cite{CaZh} and \cite{FuYuZh1}. To state them, denote by $\mathrm{OTC}_{n,k}$ the number of one-component tree-child networks of size $n$ with $k$ reticulation vertices and by $\mathrm{OTC}_n$ the (total) number of one-component tree-child networks of size $n$. Then,
\begin{equation}\label{rel-otc-l}
\mathrm{OTC}_{n,k}=\binom{n}{k}L_{n,k}
\end{equation}
and
\[
\mathrm{OTC}_n=\sum_{k=0}^{n-1}\mathrm{OTC}_{n,k}.
\]
(Note that the tree-child property implies the $k\leq n-1$ and this bound is sharp.)

\begin{pro}[\cite{CaZh,FuYuZh1}]\label{otc}
\begin{itemize}
\item[(i)] We have,
\[
\mathrm{OTC}_{n,k}=\binom{n}{k}\frac{(2n-2)!}{2^{n-1}(n-k-1)!}.
\]
\item[(ii)] As $n\rightarrow\infty$,
\[
\mathrm{OTC}_{n,k}=\frac{1}{2\sqrt{e\pi}}n^{-3/2}e^{2\sqrt{n}}\left(\frac{2}{e^2}\right)^nn^{2n}e^{-x^2/\sqrt{n}}\left(1+{\mathcal O}\left(\frac{1+\vert x\vert^3}{n}+\frac{\vert x\vert}{\sqrt{n}}\right)\right),
\]
where $k=n-\sqrt{n}+x$ and $x=o(n^{1/3})$.
\end{itemize}
\end{pro}

The second result above is a local limit theorem for the (random) number of reticulation vertices of a one-component tree-child network of size $n$ which is picked uniformly at random from all one-component tree-child networks of size $n$. It implies the following (asymptotic) counting result for $\mathrm{OTC}_n$.
\begin{cor}[\cite{FuYuZh1}]\label{asym-OTCn}
As $n\rightarrow\infty$,
\[
\mathrm{OTC}_n\sim\frac{1}{2\sqrt{e}}n^{-5/4}e^{2\sqrt{n}}\left(\frac{2}{e^2}\right)^n n^{2n}.
\]
\end{cor}

\section{Networks with Few and Many Reticulation Nodes}\label{small-large-k}

In this section, we consider $\mathrm{GTC}_{n,k}$ for small and large $k$. We start with large $k$.

As mentioned in the last section (see the sentence before Proposition~\ref{otc}), for tree-child networks, we have that $k\leq n-1$ and this bound is sharp. Clearly, this implies that $k\leq n-1$ also for galled tree-child networks. Again this bound is sharp.

\begin{lmm}\label{finite-k}
The number of reticulation vertices of a galled tree-child network of size $n$ is at most $n-1$ where this bound is sharp.
\end{lmm}
\begin{proof}
Let $N$ be a galled tree-child network of size $n$. Consider the component graph $\tilde{C}(N)$ of $N$ which is a phylogenetic tree of size $n$ (see Lemma~\ref{comp-graph-gn}). The maximal number of reticulation vertices of $N$ is achieved by placing the maximal number of arrows at all outgoing edges of internal vertices $v$ of $\tilde{C}(N)$. Note that this the degree of $v$, say $c(v)$, minus $1$, because placing arrows on all outgoing edges is not possible since $L_{c(v),c(v)}=0$. Thus, the maximal number of reticulation vertices equals
\begin{equation}\label{step-1}
\sum_{v}(c(v)-1)=\sum_{v}c(v)-(\#\ \text{internal nodes of}\ \tilde{C}(N)),
\end{equation}
where the sums run over all internal vertices of $C(N)$. By the handshake lemma,
\[
\sum_{v}c(v)=(\#\ \text{internal nodes of}\ \tilde{C}(N))+n
\]
which, by plugging into (\ref{step-1}), gives the claimed result.
\end{proof}

The proof of the last lemma also reveals the structure of maximal reticulated galled tree-child networks of size $n$: They are obtained by decompressing component graphs $\tilde{C}(N)$ that are phylogenetic trees of size $n$ with at least one leaf $\ell$ attached to every internal vertex $v$ by placing arrows on all outgoing edges of $v$ except the one leading to $\ell$. This can be translated into generating functions. Set:
\[
\mathrm{M}(z):=\sum_{n\geq 1}\mathrm{GTC}_{n,n-1}\frac{z^n}{n!},\qquad \mathrm{L}(z):=\sum_{n\geq 1}L_{n,n-1}\frac{z^{n}}{n!}=\sum_{n\geq 1}\frac{(2n-2)!}{2^{n-1}n!}z^n,
\]
where the last line follows from (\ref{rel-otc-l}) and Proposition~\ref{otc}-(i). Then, we have the following result.
\begin{lmm}
We have,
\begin{equation}\label{equ-Mz}
M(z)=z+zL'(M(z)).
\end{equation}
\end{lmm}
\begin{proof}
According to the explanation in the paragraph preceding the lemma, a maximal reticulated galled tree-child network is either a leaf or obtained from a maximal reticulated one-component tree-child network with the leafs below the reticulation vertices replaced by maximal reticulation galled tree-child networks. This translates into
\[
M(z)=z+\sum_{n\geq 1}L_{n,n-1}\frac{zM(z)^{n-1}}{(n-1)!},
\]
where the $z$ inside the sum counts the leaf which is not below the reticulation vertex and the factor $1/(n-1)!$ discards the order of the maximal reticulated galled tree-child networks (counted by $M(z)^{n-1}$) which are attached to the children below the reticulation vertices. The claimed result follows from this.
\end{proof}

Note that (\ref{equ-Mz}) is of {\it Lagrangian type}. Thus, we can obtain the asymptotics of $\mathrm{GTC}_{n,n-1}$ by applying Lagrange's inversion formula and the following result from \cite{BeRi}.

\begin{thm}[\cite{BeRi}]\label{BeRi-result}
Let $S(z)$ be a formal power series with $s_0=0, s_1\ne 0$ and $ns_{n-1}\sim\gamma s_n$. Then, for $\alpha\ne 0$ and $\beta$ real numbers,
\[
[z^n](1+S(z))^{\alpha n+\beta}\sim\alpha e^{\alpha s_1\gamma}ns_n.
\]
\end{thm}

\begin{thm}\label{max-ret-gtc}
The  number of maximal reticulated galled tree-child networks $\mathrm{GTC}_{n,n-1}$ satisfies, as $n\rightarrow\infty$,
\[
\mathrm{GTC}_{n,n-1}\sim\sqrt{e\pi}n^{-1/2}\left(\frac{2}{e^2}\right)^n n^{2n}.
\]
\end{thm}
\begin{rem}
For tree-child networks, it was proved in \cite{FuYuZh1} that $\mathrm{TC}_n=\Theta(\mathrm{TC}_{n,n-1})$. (This was a main step in the proof of (\ref{asymp-tcn}).) The above result together with Theorem~\ref{main-thm-1} shows that the same is not true for galled tree-child networks.
\end{rem}

\begin{proof}
Applying the Lagrange inversion formula to (\ref{equ-Mz}) gives
\begin{equation}\label{equ-GTCnn-1}
\mathrm{GTC}_{n,n-1}=n![z^n]M(z)=(n-1)![\omega^{n-1}](1+L'(\omega))^n.
\end{equation}
Next, by Stirling's formula, as $n\rightarrow\infty$,
\[
[z^n]L'(z)=\frac{L_{n+1,n}}{n!}=\frac{(2n)!}{2^nn!}\sim\sqrt{2}\left(\frac{2}{e}\right)^n n^n.
\]
Thus, we can apply Theorem~\ref{BeRi-result} to (\ref{equ-GTCnn-1}) with $\gamma=1/2$ and obtain that, as $n\rightarrow\infty$,
\[
\mathrm{GTC}_{n,n-1}\sim \sqrt{e}nL_{n,n-1}=\sqrt{e}n\frac{(2n-2)!}{2^{n-1}}\sim\sqrt{e\pi}n^{-1/2}\left(\frac{2}{e^2}\right)^n n^{2n}.
\]
This is the claimed result.
\end{proof}

We next consider $\mathrm{GTC}_{n,k}$ with $k$ small, i.e., the other extreme case of the number of reticulation vertices. Here, we have the following result which explains why the asymptotic expansions of $\mathrm{TC}_{n,k}$ and $\mathrm{GN}_{n,k}$ in (\ref{asymp-fixed-k}) are the same.

\begin{thm}
For fixed $k$, as $n\rightarrow\infty$,
\begin{equation}\label{asym-fixed-k}
\mathrm{GTC}_{n,k}\sim\frac{2^{k-1}\sqrt{2}}{k!}\left(\frac{2}{e}\right)^nn^{n+2k-1}.
\end{equation}
\end{thm}

The proof of this result uses ideas from \cite{FuHuYu}.
\begin{proof}
First consider galled tree-child networks of size $n$ which are obtained by decompressing phylogenetic trees of size $n$ which have all $k$ arrows on the edges from the root, i.e., the root has at least one leaf and all other children are either internal nodes or leaves (with at most $k$ internal nodes) and all internal nodes have just leaves as children. By Proposition~8 in \cite{FuHuYu}, the number of these galled tree-child network has the same asymptotics as the one on the right-hand side of (\ref{asym-fixed-k}). Moreover, these networks also dominate the asymptotics in the case of tree-child networks. Thus, the remaining galled tree-child networks are asymptotically negligible as their number is bounded above by the number of remaining tree-child networks.
\end{proof}

\section{Proof of the Main Results}\label{proof-main-thm}

In this section, we first prove Theorem~\ref{main-thm-1} and then state a result which implies Theorem~\ref{main-thm-2}.

For the proof of Theorem~\ref{main-thm-1}, we closely follow the method of proof of (\ref{asymp-gnn}) from \cite{FuYuZh2}. The main idea is to use (\ref{formula-GTCn}) to find asymptotic matching upper and lower bounds for $\mathrm{GTC}_n$.

First, for an upper bound, we pick a phylogenetic tree $\mathcal{T}$ of size $n$ (which is considered to be a component graph of a galled tree-child network of size $n$) and decompress it by picking for internal vertices $v$ of $\mathcal{T}$ {\it any} one-component tree-child network of size $c(v)$ (where the notation is as in Proposition~\ref{GNn-GuRaZh}). Since, as explained in Section~\ref{comp-graph-method}, actually only certain one-component tree-child networks are permissible, this modified decompression procedure overcounts the number of galled tree-child networks of size $n$. More precisely, we consider
\[
U_n:=\sum_{{\mathcal T}}\prod_{v}\mathrm{OTC}_{c(v)},
\]
where the first sum runs over all phylogenetic trees $\mathcal{T}$ of size $n$ and the product runs over internal vertices of $\mathcal{T}$. Then, we have $\mathrm{GTC}_n\leq U_n$. Next, set
\[
U(z):=\sum_{n\geq 1}U_n\frac{z^n}{n!},\qquad A(z):=\sum_{n\geq 1}\mathrm{OTC}_{n+1}\frac{z^n}{(n+1)!}.
\]
Then, the definition of $U_n$ implies the following result.
\begin{lmm}\label{equ-Uz}
We have,
\[
U(z)=z+U(z)A(U(z)).
\]
\end{lmm}
\begin{proof}
The networks counted by $U_n$ are either a leaf or a one-component tree-child network with $n$ leaves which are replaced by an unordered sequence of networks of the same type. This gives
\[
U(z)=z+\sum_{n\geq 2}\mathrm{OTC}_{n}\frac{U(z)^n}{n!}
\]
from which the claimed result follows.
\end{proof}

Now, we can proceed as in the proof of Theorem~\ref{max-ret-gtc} to obtain the following asymptotic result for $U_n$.
\begin{pro}\label{asymp-exp-Un}
As $n\rightarrow\infty$,
\[
U_n\sim\frac{1}{2\sqrt[4]{e}}n^{-5/4}{e}^{2 \sqrt{n}}\left(\frac{2}{e^2}\right)^n n^{2n}.
\]
\end{pro}


Next, we need a matching lower bound. Therefore, we consider (\ref{formula-GTCn}) with the first sum restricted to phylogenetic trees of the shape (where we have removed the leaf labels):

\vspace*{0.2cm}
\begin{center}
\begin{tikzpicture}
\draw (0cm,0cm) node[line width=0.03cm,inner sep=0,minimum size=0.2cm,draw,circle] (1) {};
\draw (-1.5cm,-1cm) node[line width=0.02cm,inner sep=0,minimum size=0.2cm,draw,circle] (2) {};
\draw (-0.5cm,-1cm) node[line width=0.02cm,inner sep=0,minimum size=0.2cm,draw,circle] (3) {};
\draw (0.5cm,-1cm) node[line width=0.02cm,inner sep=0,minimum size=0.2cm,draw,circle] (4) {};
\draw (1.5cm,-1cm) node[line width=0.02cm,inner sep=0,minimum size=0.2cm,draw,circle] (5) {};
\draw (-1.7cm,-2cm) node[line width=0.02cm,inner sep=0,minimum size=0.2cm,draw,circle] (6) {};
\draw (-1.3cm,-2cm) node[line width=0.02cm,inner sep=0,minimum size=0.2cm,draw,circle] (7) {};
\draw (-0.7cm,-2cm) node[line width=0.02cm,inner sep=0,minimum size=0.2cm,draw,circle] (8) {};
\draw (-0.3cm,-2cm) node[line width=0.02cm,inner sep=0,minimum size=0.2cm,draw,circle] (9) {};
\draw (-1cm,-1cm) node (10) {\small $\ldots$};
\draw (0.95cm,-1cm) node (11) {\small $\ldots$};

\draw[line width=0.02cm] (1)--(2);
\draw[line width=0.02cm] (1)--(3);
\draw[line width=0.02cm] (1)--(4);
\draw[line width=0.02cm] (1)--(5);
\draw[line width=0.02cm] (2)--(6);
\draw[line width=0.02cm] (2)--(7);
\draw[line width=0.02cm] (3)--(8);
\draw[line width=0.02cm] (3)--(9);
\draw [thick,decorate,decoration={brace,amplitude=4pt,mirror},xshift=0.4pt,yshift=-0.4pt](-1.9cm,-2.2cm)--(-0.1cm,-2.2cm) node[black,midway,yshift=-0.4cm] {\footnotesize $2j$};
\draw [thick,decorate,decoration={brace,amplitude=4pt,mirror},xshift=0.4pt,yshift=-0.4pt](0.3cm,-1.2cm)--(1.7cm,-1.2cm) node[black,midway,yshift=-0.4cm] {\footnotesize $n-2j$};
\end{tikzpicture}
\end{center}
We denote the resulting term by $L_n$. The decompression procedure from Section~\ref{comp-graph-method}, then gives the following result.
\begin{lmm}
We have,
\begin{align}
L_n&=\sum_{j=0}^{\lfloor n/2\rfloor}\binom{n}{2j}\frac{(2j)!}{j!2^j}\sum_{\ell=0}^{n-2j}\binom{n-2j}{\ell}L_{n-j,j+\ell}\nonumber\\
&=\sum_{j=0}^{\lfloor n/2\rfloor}\binom{n}{2j}\frac{(2j)!}{j!2^j}\sum_{\ell=0}^{n-2j}\binom{n-2j}{\ell}\frac{(2n-2j-2)!}{2^{n-j-1}(n-2j-\ell-1)!}.\label{exp-Ln}
\end{align}
\end{lmm}
\begin{proof}
The first equality is explained as in the proof of Lemma~9 in \cite{FuYuZh2} and the second equality follows from (\ref{rel-otc-l}) and Proposition~\ref{otc}-(i).
\end{proof}

From this result, we can deduce (matching) first-order asymptotics for $L_n$ which then together with the asymptotics of the upper bound (Proposition~\ref{asymp-exp-Un}) concludes the proof of Theorem~\ref{main-thm-1}.
\begin{pro}
As $n\rightarrow\infty$,
\[
L_n\sim\frac{1}{2\sqrt[4]{e}}n^{-5/4}{e}^{2 \sqrt{n}}\left(\frac{2}{e^2}\right)^n n^{2n}.
\]
\end{pro}

\begin{proof}[Sketch of the proof]
From Stirling's formula (similar to the proof of Proposition~\ref{otc}-(ii)),
\[
\binom{n-2j}{\ell}\frac{(2n-2j-2)!}{2^{n-j-1}(n-2j-\ell-1)!}\sim\frac{1}{2^{j+1}\sqrt{e\pi}}n^{-3/2}e^{2\sqrt{n}}\left(\frac{2}{e^2}\right)^jn^{2n-2j}e^{-x^2/\sqrt{n}},
\]
where $k=n-\sqrt{n}+x$ and this holds uniformly for small $x$ and $j$ (which both may depend on $n$). Using the Laplace method then gives,
\[
\sum_{j=0}^{n-2j}\binom{n-2j}{\ell}\frac{(2n-2j-2)!}{2^{n-j-1}(n-2j-\ell-1)!}\sim\frac{1}{2^{j+1}\sqrt{e}}n^{-5/4}e^{2\sqrt{n}}\left(\frac{2}{e^2}\right)^nn^{2n-2j}
\]
uniformly for small $j$ (which again may depend on $n$). Finally, by plugging the last relation into (\ref{exp-Ln}),
\[
L_n\sim\frac{1}{2\sqrt{e}}\left(\sum_{j\geq 0}\frac{1}{j!4^j}\right)n^{-5/4}e^{2\sqrt{n}}\left(\frac{2}{e^2}\right)^nn^{2n}
\]
which gives the claimed result.
\end{proof}

Finally, by refining the above method (see Section~6 of \cite{FuYuZh2} where the same was done for galled networks), we obtain the following result which implies our second main result (Theorem~\ref{main-thm-2}).

\begin{thm}
Let $I_n$ be the number of reticulation vertices of a random galled tree-child network of size $n$ which are not followed by a leaf and $R_n$ be the total number of reticulation vertices. Then, as $n\rightarrow\infty$,
\[
\left(I_n,\frac{R_n-n+\sqrt{n}}{\sqrt[4]{n/4}}\right)\stackrel{d}{\longrightarrow}(I,R),
\]
where $I$ and $R$ are independent with $I\stackrel{d}{=}\mathrm{Poisson}(1/4)$ and $R\stackrel{d}{=}N(0,1)$.
\end{thm}

\section{Conclusion}\label{con}

In this paper, we introduced the class of {\it galled tree-child network} which is obtained as intersection of the classes of galled networks and tree-child networks. Our reason for doing so was two-fold: (i)  Different tools have been used to prove results for galled networks and tree-child networks (\cite{FuYuZh1,FuYuZh2}); consequently, we were curious about which tools apply to the combination of these classes? (ii) It was recently proved that the number of galled networks and tree-child networks have the same first-order asymptotics when the number of reticulation vertices is fixed (\cite{ChFu,FuHuYu}). Why is that the case?

As for (i), we showed that an asymptotic counting result for galled tree-child networks (Theorem~\ref{main-thm-1}) can be obtained with the methods for galled networks, however, the result contains a stretched exponential as does the asymptotic result for tree-child networks. In addition, we showed that the number of reticulation vertices for a random galled tree-child networks is asymptotically normal (Theorem~\ref{main-thm-2}), whereas the limit laws of the same quantities for galled networks and tree-child networks were discrete. As for (ii), we showed that the number of galled tree-child networks also satisfies the same first order asymptotics when the number of reticulation vertices is fixed. This explains the previous results from \cite{ChFu,FuHuYu}.

{\small\paragraph{\small\bf Acknowledgments.} The authors acknowledge partial supported by NTCS, Taiwan under the grants NSTC-111-2115-M-004-002-MY2 (YSC, MF) and NSTC-110-2115-M-017-003-MY3 (GRY).}


\vspace{-0.2cm}
\begin{thebibliography}{99}
\bibitem{BeRi} E. A. Bender and L. B. Richmond (1984). An asymptotic expansion for the coefficients of some power series. II. Lagrange inversion, {\it Discrete Math.}, {\bf 50:2-3}, 135--141.
\bibitem{BoGaMa} M. Bouvel, P. Gambette, M. Mansouri (2020). Counting phylogenetic networks of level $1$ and level $2$, {\it J. Math. Biol.}, {\bf 81:6-7}, 1357--1395.
\bibitem{CaZh} G. Cardona and L. Zhang (2020). Counting and enumerating tree-child networks and their subclasses, {\it J. Comput. System Sci.}, {\bf 114}, 84--104.
\bibitem{ChFu} Y.-S. Chang and M. Fuchs. Counting phylogenetic networks with few reticulation vertices: galled and reticulation-visible networks, arXiv:2401.08958.
\bibitem{ChFuLiWaYu} Y.-S. Chang, M. Fuchs, H. Liu, M. Wallner, G.-R. Yu (2022). Enumerative and distributional results for $d$-combining tree-child networks, arXiv:2209.03850.
\bibitem{ElFaWa} A. Elvey Price, W. Fang, M. Wallner. Compacted binary trees admit a stretched exponential, {\it J. Comb. Theory Ser. A}, {\bf 177}, Article 105306.
\bibitem{FuGiMa1} M. Fuchs, B. Gittenberger, M. Mansouri (2019). Counting phylogenetic networks with few reticulation vertices: tree-child and normal networks, {\it Australas. J. Combin.}, {\bf 73:2}, 385--423.
\bibitem{FuGiMa2} M. Fuchs, B. Gittenberger, M. Mansouri (2021). Counting phylogenetic networks with few reticulation vertices: exact enumeration and corrections, {\it Australas. J. Combin.}, {\bf 82:2}, 257--282.
\bibitem{FuHuYu} M. Fuchs, E.-Y. Huang, G.-R. Yu (2022). Counting phylogenetic networks with few reticulation vertices: a second approach, {\it Discrete Appl. Math.}, {\bf 320}, 140--149.
\bibitem{FuYuZh1} M. Fuchs, G.-R. Yu, L. Zhang (2021). On the asymptotic growth of the number of tree-child networks, {\it European J. Combin.}, {\bf 93}, 103278, 20pp.
\bibitem{FuYuZh2} M. Fuchs, G.-R. Yu, L. Zhang (2022). Asymptotic enumeration and distributional properties of
galled networks, {\it J. Comb. Theory Ser. A.}, {\bf 189}, 105599, 28 pages.
\bibitem{GuRaZh} A. D. M. Gunawan, J. Rathin, L. Zhang (2020). Counting and enumerating galled networks, {\it Discrete Appl. Math.}, {\bf 283}, 644--654.
\bibitem{KoPoKuWi} S. Kong, J. C. Pons, L. Kubatko, K. Wicke (2022). Classes of explicit phylogenetic networks and their biological and mathematical significance, {\it J. Math. Biol.}, {\bf 84}, Paper: 47.
\bibitem{Ma}M. Mansouri (2022). Counting general phylogenetic networks, {\it Australas. J. Combin.}, {\bf 83}, 40--86.
\bibitem{PoBa} M. Pons and J. Batle (2021). Combinatorial characterization of a certain class of words and a conjectured connection with general subclasses of phylogenetic tree-child networks, {\it Sci. Rep.}, {\bf 11}, Article number: 21875.
\end{thebibliography}
\end{document}